\documentclass[reqno,11pt]{amsart}

\usepackage{euscript}
\usepackage{enumerate}
\usepackage{amsmath}
\usepackage{mathrsfs}
\usepackage{mathtools}
\usepackage{verbatim}
\usepackage{enumitem}
\usepackage[initials]{amsrefs}
\usepackage{amssymb}
\usepackage{comment}
\usepackage{color}
\usepackage{amsthm}
\DeclareMathOperator*{\essinf}{ess\,inf}

\usepackage{bbm}
\usepackage{hyperref}

\theoremstyle{plain}
\newtheorem{thm}{Theorem}[section]
\newtheorem{prop}[thm]{Proposition}
\newtheorem{cor}[thm]{Corollary}

\newtheorem{lemma}[thm]{Lemma}

\newtheorem{defi}[thm]{Definition}

\theoremstyle{remark}
\newtheorem{rem}[thm]{Remark}

\theoremstyle{definition}
\newtheorem{example}[thm]{Example}

 \newcommand\Cpx{{\mathbb C}}

 \newcommand\Mcal{{\mathcal{M}}}

 \newcommand\restrict{{\upharpoonright}}

\newcommand{\BH}{\mathcal{B}(\mathcal{H})}
\newcommand{\M}{\mathcal{M}}
\newcommand{\K}{\mathcal{K}}
\newcommand{\Hsp}{\mathcal{H}}
\newcommand{\C}{\mathbb{C}}

\newcount\theTime
\newcount\theHour
\newcount\theMinute
\newcount\theMinuteTens
\newcount\theScratch
\theTime=\number\time
\theHour=\theTime
\divide\theHour by 60
\theScratch=\theHour
\multiply\theScratch by 60
\theMinute=\theTime
\advance\theMinute by -\theScratch
\theMinuteTens=\theMinute
\divide\theMinuteTens by 10
\theScratch=\theMinuteTens
\multiply\theScratch by 10
\advance\theMinute by -\theScratch

\def\today{{\number\day\space
 \ifcase\month\or
  January\or February\or March\or April\or May\or June\or
  July\or August\or September\or October\or November\or December\fi
 \space\number\year}}

\newcommand\HEu{{\EuScript H}}                   

\newcommand\norm[1]{\ensuremath{\left\vert\left\vert #1 \right\vert\right\vert}}

\begin{document}

\title[]{Decomposability of Operators in Type $\mathrm{I}_k$ von Neumann Algebras}

\author[Karri]{Ajay Kumar Karri}
 \address{Ajay Kumar Karri, Department of Mathematics, Texas A\&M University, College Station, TX 77843-3368, USA.}
 \email{ajay85426@tamu.edu}

\subjclass[2010]{47C15, 47A11, 47B40}

\keywords{type I von Neumann algebra, Decomposability, spectrality, Norm convergence property}

\begin{abstract}
Let $\mathcal{H}$ be a complex Hilbert space and $\BH$ be the algebra of all bounded linear operators on $\mathcal{H}$. For $A \in \BH$, we refer to the sequence $\{|A^{n}|^{1/n}\}_{n\in\mathbb{N}}$ as the normalized power sequence of $A$. In this article, we study the norm convergence property, i.e. convergence of normalized power sequence in the norm topology for operators belonging to type $\mathrm{I}_k$ von Neumann algebras acting on a separable complex Hilbert space. By utilizing continuous upper-triangular forms via unitary conjugations, we construct specific projection-valued families to prove that every operator in a type $\mathrm{I}_k$ von Neumann algebra is decomposable. As a consequence, this immediately establishes that every such operator possesses the norm convergence property, extending recent results known for matrices with complex-valued entries, compact operators on a separable Hilbert space, spectral operators, and Riesz operators. Finally, we provide a counterexample within the type $I_\infty$ factor $\mathcal{B}(l^2(\mathbb{N}))$ to demonstrate that this convergence property generally fails when the dimension $k$ is infinite.
\end{abstract}

\date{\today}

\maketitle

\section{Introduction}
 Let $\mathcal{H}$ be a complex Hilbert space and $\BH$ denote the algebra of all bounded linear operators on $\mathcal{H}$. For an operator $A \in \mathcal{B}(\mathcal{H})$, Nayak and Shekhawat \cite{NS25} referred to the sequence $\{|A^n|^{1/n}\}_{n \in \mathbb{N}} = \{((A^*)^n A^n)^{1/2n}\}_{n \in \mathbb{N}}$ as the \emph{normalized power sequence} (NPS) of $A$.

Following Dykema, Noles, and Zanin \cite{DNZ17}, an operator $A$ in a $C^*$-algebra is said to possess the \emph{norm convergence property} if its normalized power sequence $\{|A^{n}|^{1/n}\}_{n\in\mathbb{N}}$ converges in the norm topology to a positive operator.

The convergence of normalized power sequences has recently drawn significant interest across various operator-algebraic settings:
\begin{itemize}
    \item Haagerup and Schultz \cite{HS09} initially investigated the strong operator topology (SOT) convergence of the NPS for operators in tracial von Nuemann algebras.
    \item Nayak \cite{Nay23} established that every matrix in $M_k(\mathbb{C})$ possesses the norm convergence property.
    \item Subsequently, Bhat and Bala \cite{BB24} extended the norm convergence property to compact operators acting on separable complex Hilbert spaces, while Nayak and Shekhawat proved norm convergence for spectral operators \cite{NS24} as well as Riesz operators \cite{NS26}.
    \item In the setting of von Neumann algebras, Nayak and Shekhawat \cite{NS25} proved that the NPS of every operator in a type $\text{I}$ Murray--von Neumann algebra converges in the measure topology(cf.\ \cite{Nel74}).
\end{itemize}

Despite these developments, whether full norm convergence holds for general non-compact operators in finite type $\text{I}$ von Neumann algebras remained open. In \cite[Theorem~4.3]{DNZ17}, Dykema, Noles, and Zanin showed that in a finite tracial von Neumann algebra, operator \emph{decomposability} implies the norm convergence property.

The primary objective of this article is to settle this problem for type $\text{I}_k$ von Neumann algebras by establishing the decomposability of all operators in this algebra. The main result of this paper is as follows:

\begin{thm}\label{main theorem}
 Let $\mathcal{M}$ be a type $\mathrm{I}_k$ ($k \in \mathbb{N}$) von Neumann algebra acting on a separable Hilbert space $\mathcal{H}$, and $A \in \mathcal{M}$. The operator $A$ is decomposable and, consequently, possesses the norm convergence property.
\end{thm}

To prove Theorem \ref{main theorem}, we exploit the canonical structure of type $\text{I}_k$ von Neumann algebras as $M_k(C(S))$ over a hyper-Stonean space $S$. By refining continuous upper-triangular representations via unitary conjugations, we construct pointwise-ordered diagonal forms. This enables the construction of explicit families of invariant projections that closely mirror the Haagerup--Schultz projections. By verifying the spectral inclusion conditions for restrictions to these invariant subspaces, we deduce decomposability via equivalent spectral criteria.

Additionally, we present two structural non-examples:
\begin{enumerate}
    \item Using the Dunford decomposition, we construct an operator in a type $\text{I}_3$ von Neumann algebra that fails to be a spectral operator, demonstrating that our main result strictly encompasses non-spectral operators.
    \item We demonstrate that this convergence property fails when $k=\infty$ by providing a counterexample in the type $\text{I}_\infty$ factor $\mathcal{B}(\ell^2(\mathbb{N}))$.
\end{enumerate}

\noindent\textbf{Organization of the Paper.} The rest of this article is organized as follows. In Section~2, we provide necessary background on Brown measures, Haagerup--Schultz projections, continuous upper-triangular forms, and decomposability. Section~3 contains the construction of the key invariant projections and the proof of Theorem \ref{main theorem}. Section~4 details the construction of a non-spectral operator in a type $\text{I}_3$ von Neumann algebra and an operator in a type $\text{I}_\infty$ factor that does not possess norm convergence property.

\section{Preliminaries and notations}
\label{sec:background}
Throughout the article, the following notation and language will be used: $\Mcal$ will be a von Neumann algebra of operators on a complex Hilbert space $\mathcal{H}$. Unless otherwise specified, $A$ will be an element of $\Mcal$ and
then $\sigma(A)$ will denote the spectrum of $A$.
Finally, we use the standard notations:
$\Cpx$ is the complex plane, $\mathbb{D}$ is the open unit disk in $\Cpx$ centered at the origin, and $\mathbb{T}$ is the unit circle, namely,
the boundary of $\mathbb{D}$.

We denote the set of all $k \times k$ matrices over an algebra $\mathcal{A}$ by $M_k(\mathcal{A})$ and the identity matrix by $I_k$, or just $I$ if the dimension of the matrix is clear. For $k \in \mathbb{N}$, $[k]$ denotes the set $\{1, \ldots, k\}$. For $i, j \in [k]$, the elementary matrix in $M_k(\mathcal{A})$ with $(i,j)^{\text{th}}$ entry equal to the unity of $\mathcal{A}$ and rest of the entries equal to the zero element of $\mathcal{A}$, is denoted by $E_{ij}$.~We note the following basic fact about the multiplication of elementary matrices:
\begin{equation*}
E_{ij} E_{k\ell} = \delta_{jk} E_{i\ell} \quad \text{for } i, j, k, \ell \in [n].
\end{equation*}

Let $(X, \mu)$ be a measure space, and $A = (A_{ij}) \in M_k(L^\infty(X, \mu))$. For any $1 \le l \le k$, we denote:
\begin{itemize}
    \item $A^{(l)} = (A_{ij})_{1 \le i,j \le l} \in M_l(L^\infty(X, \mu))$ as the leading principal submatrix formed by the first $l$ rows and first $l$ columns of $A$;
    \item $A_{(l)} = (A_{ij})_{k-l+1 \le i,j \le k} \in M_l(L^\infty(X, \mu))$ as the trailing principal submatrix formed by the last $l$ rows and last $l$ columns of $A$;
    \item $A|_Y = (A_{ij}|_Y)_{1 \le i,j \le k} \in M_k(L^\infty(Y, \mu|_Y))$ as the restriction of $A$ to a measurable subset $Y \subseteq X$.
\end{itemize}

\subsection{Brown measure and Haagerup-Schultz projections} 
  In this subsection we consider $\M$ is a finite von Neumann algebra acting on the Hilbert space $\Hsp$ and with faithful, tracial state $\tau$. 
  
  In \cite{Br86}, L.~Brown introduced a generalization of the spectral distribution measure for not necessarily normal operators in tracial von Neumann algebras.
\begin{thm}\label{thm:brnmeas}
Let $T\in \mathcal{M}$.
Then there exists a unique probability measure $\nu_T$ such that for every $\lambda\in\Cpx$, 
$$\int_{[0,\infty)}\log(x)\,d\mu_{|T-\lambda|}(x)=\int_\Cpx \log|z-\lambda|\,d\nu_T(z),$$
where for a positive operator $S$, $\mu_S$ denotes the spectral distribution measure $\tau \circ E$, where $E$ is the spectral measure of $S$.
\end{thm}
The measure $\nu_T$ in Theorem \ref{thm:brnmeas} is called the {\em Brown measure} of $T$.
If $T$ is normal, then $\nu_T$ equals the spectral distribution of $T$.
The following is the main result from \cite{HS09}. It provides projections that split the operator $T$ according to the Brown measure.
\begin{thm}\cite[Theorem 1.1]{HS09}\label{HS main thm}
    For every $T \in \M$ and every Borel set $B \subseteq \C$ there is a largest closed, $T$-invariant subspace, $\K = \K_T(B)$, affiliated with $\M$, such that the Brown measure of $T|_\K$, $\mu_{T|_\K}$, is concentrated on $B$. Moreover, $\K$ is hyperinvariant for $T$, and if $P = P_T(B) \in \M$ denotes the projection onto $\K$, then
    \begin{enumerate}[label=(\roman*),leftmargin=40pt]
    \item $\tau(P) = \mu_T(B)$,
    \item the Brown measure of $P^\perp T P^\perp$, considered as an element of $P^\perp \M P^\perp$, is concentrated on $\C \setminus B$.
    \end{enumerate}
\end{thm}
Next, we are going to recall certain hyperinvariant subspaces for $T$.
\begin{defi}\cite[Definition 3.1]{HS09}\label{HS E and F definition}
\begin{enumerate}[label=(\roman*),leftmargin=40pt]
    \item For $T \in \M$ and $r > 0$, let $E(T, r)$ denote the set of $\xi \in \Hsp$, for which there exists a sequence $(\xi_n)_{n=1}^\infty$ in $\Hsp$, such that
    \begin{equation}
        \lim_{n \to \infty} \|\xi_n - \xi\| = 0 \quad \text{and} \quad \limsup_{n \to \infty} \|T^n \xi_n\|^{\frac{1}{n}} \le r.
    \end{equation}
    \item For $T \in \M$ and $r > 0$, let $F(T, r)$ denote the set of $\eta \in \Hsp$, for which there exists a sequence $(\eta_n)_{n=1}^\infty$ in $\Hsp$, such that
    \begin{equation}
        \lim_{n \to \infty} \|T^n \eta_n - \eta\| = 0 \quad \text{and} \quad \limsup_{n \to \infty} \|\eta_n\|^{\frac{1}{n}} \le \frac{1}{r}.
    \end{equation}
\end{enumerate}
\end{defi}
\begin{lemma}\cite[Lemma 3.6]{HS09}\label{HS E and F relation}
    For $T \in \M$ and $s > r > 0$ one has that $E(T, r) \perp F(T^*, s)$.
\end{lemma}

\begin{prop}\cite[Corollary 7.19]{HS09}\label{HS E and F result}
    For every $T \in \M$, every $\lambda \in \C$ and every $r > 0$ one has:
\begin{enumerate}[label=(\roman*),leftmargin=40pt]
    \item $E := E(T - \lambda \mathbf{1}, r)$ is the largest, closed $T$-invariant subspace affiliated with $\M$, such that $\operatorname{supp}(\mu_{T|_E}) \subseteq \overline{B(\lambda, r)}$.
    \item $F := F(T - \lambda \mathbf{1}, r)$ is the largest, closed $T$-invariant subspace affiliated with $\M$, such that $\operatorname{supp}(\mu_{T|_F}) \subseteq \C \setminus B(\lambda, r)$.
\end{enumerate}
Therefore, $E(T - \lambda \mathbf{1}, r)=\K_T(\overline{B(\lambda, r)})$ and $F(T - \lambda \mathbf{1}, r)=\K_T(\C \setminus B(\lambda, r))$.
\end{prop}

The results about Brown measure and Haagerup-Schultz projections in the following lemma are basic and easy to prove.
\begin{lemma}\label{lem:basicHS}
Let $A\in\Mcal$.
Then for any $\lambda\in\mathbb{C}$ and any Borel set $B\subseteq\mathbb{C}$, letting $B^*$ denote the image of $B$ under complex conjugation, we have 
\begin{enumerate}[label=(\roman*),leftmargin=30pt]
\item $\nu_{(A-\lambda)}(B) = \nu_A(B+\lambda)$\label{i}
\item $\nu_{A^*}(B) = \nu_A(B^*)$\label{ii}
\item $P(A-\lambda,B) = P(A,B+\lambda)$\label{iii}
\item\label{it:PT*} $P(A^*,B) = 1-P(A,\mathbb{C}\setminus B^*)$.\label{iv}
\end{enumerate}
\end{lemma}

The following lemma shows that the spectra of compressions onto invariant subspaces preserve inclusion under sub-projections.

\begin{lemma}\label{p,q spectrum proj}
    Let $p, q$ be invariant projections of $A \in \Mcal$ with $p \le q$. Then $\sigma(Ap) \subseteq \sigma(Aq)$ and $\sigma((1-q)A) \subseteq \sigma((1-p)A)$ where the spectra are computed in the compressions of $\Mcal$ by the projections $p$, $q$, $1-q$ and $1-p$ respectively.
\end{lemma}

\begin{proof}
     If $\lambda \notin \sigma(Aq)$, then there exists $ \tilde{A} \in M$ such that
\begin{align*}
    (q \tilde{A} q)(Aq - \lambda q) &= q \\
\implies    q \tilde{A} q (A - \lambda) q &= q \\
\implies    q \tilde{A} (A - \lambda) q &= q \\
\implies    p q \tilde{A} (A - \lambda) q p &= p q p \\
\implies    p \tilde{A} (A - \lambda) p &= p \\
\implies    (p \tilde{A} p) \cdot (A - \lambda) p &= p.
\end{align*}
Then $\lambda \notin \sigma(Ap)$ which implies $\sigma(Ap) \subseteq \sigma(Aq)$. Similarly, we can show that $\sigma((1-q)A) \subseteq \sigma((1-p)A)$.
\end{proof}

\subsection{Upper-triangular forms in type $\mathrm{I}_k$ (for $ k \in \mathbb{N}$) von Neumann algebras}
\label{subsec:UT}

Every type $\mathrm{I}_k$ von Neumann algebra with center $\mathscr{C}$ is of the form $M_n(\mathscr{C})$ (see \cite[Theorem 6.6.5]{KR97}), and $\mathscr{C}$ is $*$-isomorphic to $C(S)$ where $S$ is extremely disconnected compact Hausdorff space (see \cite[Theorem 5.2.1]{KR83}). For convenience, we recall the definition of such underlying topological spaces.

\begin{defi}(Stonean space)
A compact Hausdorff space $S$ is said to be a Stonean space if the closure of each open set in $S$ is open in $S$. Such spaces are also known as extremally disconnected spaces in the literature.

A (hyper)-Stonean space is a special type of Stonean space characterized by having a dense set of supports coming from normal measures.
\end{defi}

We can also represent $\mathscr{C}$ as $L^\infty(X, \mu)$, where $X$ is a compact subset of $\mathcal{R}$ and $\mu$ is a regular Borel probability measure on $X$ (see \cite[Corollary II.2.9]{Dav96}). Thus, every type $\mathrm{I}_k$ von Neumann algebra is of the form $M_k(C(S))$ for some Stonean space (in fact (hyper)-Stonean space) $S$, as well as of the form $M_k(L^\infty(X, \mu))$ for some compact set $X$ and a regular Borel probability measure $\mu$ on $X$. There is a canonical $*$-isomorphism between $L^\infty(X, \mu)$ and $C(S)$ where the points in the (hyper)-Stonean space $S$ correspond to ultrafilters of measurable subset in $X$ modulo $\mu$-null set. We will often use $C(S)$ and $L^{\infty}(X,\mu)$ interchangeably.

Let $X$ be a measurable space. Let $A \in L^\infty(X,\mu; M_k(\mathbb{C}))$. For $x \in X$, let $a_{ij}(x)$ be the $(i,j)^{\text{th}}$ entry of the matrix $A(x) \in M_k(\mathbb{C})$. Clearly $a_{ij} : X \to \mathbb{C}$, defined by $x \mapsto a_{ij}(x)$, belongs to $L^\infty(X; \mathbb{C})$. This gives a natural mapping
\begin{equation*}
    \varphi : L^\infty(X,\mu; M_k(\mathbb{C})) \to M_k(L^\infty(X,\mu; \mathbb{C}))
\end{equation*}
defined by
\begin{equation*}
    A \longmapsto [a_{ij}]_{i,j=1}^k.
\end{equation*}
Note that $\varphi$ is a $*$-isomorphism.~So, $L^\infty(X,\mu; M_k(\mathbb{C})) \cong M_k(L^\infty(X,\mu))$. We will often view these $*$-algebras interchangeably. Similarly, one can view $M_n(C(S))$ and $C(S,M_k(\mathbb{C}))$ $*$-isomorphic in a natural way. 

Now, we will recall some properties of operator algebra on $C(S)$ and $L^\infty(X, \mu)$. Let $f_1, f_2 \in C(S)$, then $f_1 \geqslant f_2$ (in the sense of operators) iff for every $s \in S$, $f_1(s) \geqslant f_2(s)$. Also let $f_1, f_2 \in L^\infty(X, \mu)$, then $f_1 \geqslant f_2$ (in the sense of operators) iff $f_1(x) \geqslant f_2(x)$ almost everywhere $x \in X$ (i.e. $f_1 \geqslant f_2$ a.e. on $X$).

\begin{lemma}\label{inc order in two spaces}
Let $\pi: C(S) \longrightarrow L^\infty(X, \mu)$ be a $*$-isomorphism. Then for any $f_1, f_2 \in C(S)$ such that for every $s \in S$, $|f_1(s)| \leqslant |f_2(s)|$, then $|\pi(f_1)| \leqslant |\pi(f_2)|$ a.e. on $X$.
\end{lemma}

\begin{proof}
Since, the $*$-isomorphism preserves order. Let $f_1, f_2 \in C(S)$ then
\begin{align*}
f_1^* f_1 \leqslant f_2^* f_2 &\iff \pi(f_1^* f_1) \leqslant \pi(f_2^* f_2) {\quad \text{(in the sense of operators)}} \\
\Rightarrow \quad |f_1|^2 \leqslant |f_2|^2 &\iff |\pi(f_1)|^2 \leqslant |\pi(f_2)|^2 \\
\Rightarrow \quad |f_1| \leqslant |f_2| &\iff |\pi(f_1)| \leqslant |\pi(f_2)| 
\end{align*}
Hence, for every $s \in S$, $|f_1(s)| \leqslant |f_2(s)|$ iff $|\pi(f_1)| \leqslant |\pi(f_2)|$ a.e. on $X$. 
\end{proof}

In \cite[Theorem 2]{DP63}, Deckard and Pearcy proved that for a Stonean space $S$, every matrix in $M_n(C(S))$ may be transformed into upper-triangular form via unitary conjugation. Since (hyper)-Stonean spaces are  Stonean spaces, this result is true for $M_n(C(S))$, where $S$ is a (hyper)-Stonean space. We now show that the upper triangular form can be chosen such that its diagonal entries are ordered pointwise by magnitude.

\begin{prop}\label{Unitarization 1}
   Let $S$ be a Stonean space, $k \in \mathbb{N}$ and $f_1, \dots, f_k \in C(S)$. Then there exist $ g_1, \dots, g_k \in C(S)$ such that for every $s \in S$,
$$|g_1(s)| \le  |g_2(s)| \le \dots \le |g_k(s)|$$
and $g_1(s), g_2(s), \dots, g_k(s)$ is a re-ordering of $f_1(s), \dots, f_k(s)$.
\end{prop} 

\begin{proof}
We proceed by  induction on $k \ge 1$. For $k=1$, the statement holds trivially.
 
$\mathbf{Base ~ Case}$ ($k=2$): We define
\begin{align*}
    U_1 &:= \{ s \in S : |f_1(s)| < |f_2(s)| \}, \\
    U_2 &:= \{ s \in S : |f_1(s)| > |f_2(s)| \}, \\
    E   &:= \{ s \in S : |f_1(s)| = |f_2(s)| \}.
\end{align*}
Note that $S = U_1 \sqcup U_2 \sqcup E$, and $U_1, U_2$ are open sets. Since $S$ is an extremely disconnected (Stonean) space, the closures $\overline{U}_1$ and $\overline{U}_2$ are disjoint clopen sets satisfying
\begin{align*}
    V_1 &:= \overline{U}_1 \subseteq \{ s \in S : |f_1(s)| \le |f_2(s)| \}, \\
    V_2 &:= \overline{U}_2  \subseteq \{ s \in S : |f_1(s)| \ge |f_2(s)| \}, \\
    V_3 &:= S \setminus (V_1 \cup V_2) \subseteq E.
\end{align*}
The sets $V_1, V_2$, and $V_3$ are clopen and form a partition of $S$. Define
\begin{align*}
    g_{1} &:= f_1 \cdot \mathbbm{1}_{V_1 \sqcup V_3} + f_2 \cdot \mathbbm{1}_{V_2}, \\
    g_{2} &:= f_2 \cdot \mathbbm{1}_{V_1 \sqcup V_3} + f_1 \cdot \mathbbm{1}_{V_2}.
\end{align*}
Then $g_1, g_2 \in C(S)$ satisfy $|g_1(s)| \le |g_2(s)|$ for all $s \in S$, completing the base case.

$\mathbf{Inductive ~ Step}$: Assume the statement holds for any $k$ continuous functions on $S$. Let $f_1, f_2, \dots, f_{k+1} \in C(S)$. ~ Applying the induction hypothesis to $\{f_1, \dots, f_{k}\}$, we obtain a sorted sequence of continuous functions $\tilde{f}_1, \dots, \tilde{f}_{k}$ such that for every $s \in S$,
$$|\tilde{f}_1(s)| \le |\tilde{f}_2(s)| \le \dots \le |\tilde{f}_{k}(s)|$$
and $\tilde{f}_1(s), \tilde{f}_2(s), \dots, \tilde{f}_{k}(s)$ is a re-ordering of $f_1(s), \dots, f_{k}(s)$. 

Next, applying the base case to $\tilde{f}_{k}$ and $f_{k+1}$, we obtain continuous functions $\tilde{g}_{k}$ and ${g}_{k+1}$ such that for every $s \in S$,
$$ |\tilde{g}_{k}(s)| \le |g_{k+1}(s)| $$ and $ \, \tilde{g}_{k}(s), \, g_{k+1}(s)$ is a re-ordering of  \,$\tilde{f}_{k}(s), \,  f_{k+1}(s)$. Consequently for every $s \in S$, $\tilde{f}_1(s), \, \tilde{f}_2(s), \dots, \, \tilde{f}_{k-1}(s), \, \tilde{g}_{k}(s), \, g_{k+1}(s)$ is a re-ordering of $f_1(s), \, f_2(s),$ $\dots, f_{k+1}(s)$ and $|\tilde{f}_{i}(s)| \le |g_{k+1}(s)|$ for all $1 \le i \le k-1$. Applying the induction hypothesis once more to $\{\tilde{f}_1, \tilde{f}_2, \dots, \tilde{f}_{k-1}, \tilde{g}_k\}$, we obtain a sorted sequence of continuous functions $g_1, g_2, \dots, g_{k}$ such that for every $s \in S$,
$$|g_1(s)| \le |g_2(s)| \le \dots \le |g_{k}(s)|$$
and $g_1(s), g_2(s), \dots, g_{k}(s)$ is a re-ordering of $\tilde{f}_1(s), \tilde{f}_2(s), \dots, \tilde{f}_{k-1}(s), \tilde{g}_k (s)$.  Consequently for every $s \in S$, $g_1(s),$ $g_2(s), \dots, g_{k+1}(s)$ is a re-ordering of $f_1(s), f_2(s), \dots, f_{k+1}(s)$  and 
$$|g_1(s)| \le |g_2(s)| \le \dots \le |g_{k+1}(s)|,$$
completing the induction.
\end{proof}

\begin{lemma}\cite[Lemmas 3.1 and 3.2]{DP63}\label{lem: Decard}
Let $S$ be a Stonean space, $A \in M_k(C(S))$, and $\lambda \in C(S)$ such that $\det(A(s) - \lambda(s)I) = 0$ for all $s \in S$. Then there exists a unitary matrix $U \in M_k(C(S))$ such that
$$B = U^* A U = (b_{ij})_{1 \le i, j \le k},$$ where $b_{11} = \lambda$ and $b_{i1} = 0$ for all $2 \le i \le k$.
\end{lemma}
\begin{proof}[Proof Sketch]
The proof follows an argument analogous to those of \cite[Lemmas 3.1 and 3.2]{DP63}. Specifically, the construction in \cite[Lemma 3.1]{DP63} allows $\lambda$ to be chosen as an arbitrary eigenvalue function, from which it follows that the leading entry $b_{11}$ in \cite[Lemma 3.2]{DP63} coincides precisely with $\lambda$.
\end{proof}

\begin{prop}\label{lem:re-ord}
    Let $S$ be a Stonean space, and $A \in M_k(C(S))$. Then there exist a unitary $U \in M_k(C(S))$ such that $B = U^* A U = (b_{ij})_{1 \le i, j \le k}$ is upper triangular, (namely, $b_{ij} = 0$ if $i > j$) and for every $s \in S$
$$|b_{11}(s)| \le |b_{22}(s)| \le \dots \le |b_{kk}(s)|.$$
\end{prop}

\begin{proof} 
 We proceed by  induction on $k \ge 1$. For $k=1$, the statement holds trivially.
 
 $\mathbf{Base ~ case}$ ($k=2$): By \cite[Theorem 1]{DP63}, there exist functions $\lambda_1, \lambda_2 \in C(S)$ satisfying $\det(\lambda_i I - A) = 0$
for $i \in \{1,2\}$ and 
$$\det(\lambda I - A) =  (\lambda - \lambda_1) \, (\lambda - \lambda_2) \qquad \text{ for all } \lambda \in \mathbb{C}.$$
By Proposition \ref{Unitarization 1}, we obtain reordered continuous functions $\lambda_{(1)}, \lambda_{(2)} \in C(S)$ such that $|\lambda_{(1)}(s)|  \le  |\lambda_{(2)}(s)| $ for every $s \in S$. Lemma \ref{lem: Decard} ensures the existence of a unitary matrix $U \in M_2(C(S))$ such that $B=U^* \, A \, U = (b_{ij})_{1 \le i,j \le 2}$ is upper triangular with $b_{11}=\lambda_{(1)}$. Since 
\begin{align*}
    \det(\lambda I - B) =& \, \det(\lambda I - A)\\
    \implies (\lambda - \lambda_{(1)})(\lambda - b_{22})=& \, (\lambda - \lambda_{(1)})(\lambda - \lambda_{(2)}),
\end{align*}
 it follows that $b_{22}=\lambda_{(2)}$, completing the base case.
 
$\mathbf{Inductive ~ step}$: Assume the statement holds for matrices in $ M_k(C(S))$ for some $k \ge 2$. Let $A \in  M_{k+1}(C(S))$.~ By \cite[Theorem 1]{DP63} and Proposition \ref{Unitarization 1}, there exist functions $\lambda_1, \dots, \lambda_{k+1} \in C(S)$ representing a pointwise non-decreasing ordering of the eigenvalues, i.e.
\[
  |\lambda_{(1)}(s)|  \le  |\lambda_{(2)}(s)| \le \dots \le |\lambda_{(k+1)}(s)| \qquad \text{ for all }  s \in S,  
\]
and
\[
  \det(\lambda I - A) = \prod_{i=1}^{k+1} (\lambda - \lambda_i) \qquad \text{ for all } \lambda \in \mathbb{C}.
\] 
By Lemma \ref{lem: Decard}, there exists a unitary element $U_1 \in M_{k+1}(C(S))$ such that
$$B_1 = U_1^* A U_1 = \begin{bmatrix} \lambda_{(1)} & w \\ 0 & B' \end{bmatrix},$$
where $w \in {C(S)}^k$ and $B' \in M_{k}(C(S))$. By the induction hypothesis applied to $B'$, there exists a unitary matrix $U' \in M_{k}(C(S))$ such that $\tilde{B} = (U')^* B' U' = (\tilde{b}_{ij})_{1 \le i,j \le k}$ is upper triangular matrix and satisfies
$ |\tilde{b}_{11} (s)|  \le \dots \le |\tilde{b}_{k, k} (s)| $ for all $s \in S$. Construct the unitary operator 
\[
 U= U_1 \, \begin{bmatrix} 1 & 0 \\ 0 & U' \end{bmatrix} \in  M_{k+1}(C(S)).
\]
Then $B=U^* A U$ is upper triangular of the form 
\[
  B= \begin{bmatrix} \lambda_{(1)}  & w U' \\ 0 & \tilde{B} \end{bmatrix}.
\]
Comparing the characteristic polynomials, $\tilde{b}_{ii} = \lambda_{(i+1)}$ for each $1 \le i \le k$. Thus the diagonal entries of $B$ satisfy $ |\lambda_{(1)}(s)| \le |\tilde{b}_{11}(s)| \le \dots \le |\tilde{b}_{k, k}(s)|$ for all $s \in S$, completing the induction.
\end{proof}
 We similarly prove for the case when the absolute values of diagonal entries are non-increasing pointwise.
\begin{cor}\label{cor: re-ord}
     Let $S$ be a Stonean space, and $A \in M_k(C(S))$. Then there exist a unitary $U \in M_k(C(S))$ such that $B = U^* A U = (b_{ij})_{1 \le i, j \le k}$ is upper triangular, (namely, $b_{ij} = 0$ if $i > j$) and for every $s \in S$
$$|b_{11}(s)| \ge |b_{22}(s)| \ge \dots \ge |b_{kk}(s)|.$$
\end{cor}
\begin{proof}
Let $J= \sum_{i=1}^k E_{i,k-i+1} \in M_k(C(S))$ be the anti-diagonal permutation unitary matrix (having $\mathbbm{1}$ on its anti-diagonal and $0$ elsewhere). Applying Lemma 2.10 to the matrix $A' = J A^* J^* \in M_k(C(S))$ yields a unitary $V \in M_k(C(S))$ such that $T = V^* A' V$ is upper triangular with pointwise non-decreasing diagonal entries 
\[
|t_{11}(s)| \le |t_{22}(s)| \le \dots \le |t_{kk}(s)| \quad \text{for all } s \in S.
\]
Setting $U = J^* V J \in M_k(C(S))$, the transformed matrix $B = U^* A U = J T^* J^*$ is upper triangular, and its diagonal entries satisfy $b_{ii} = \bar{t}_{k-i+1, k-i+1}$ for $1 \le i \le k$. Consequently, we obtain
\[
|b_{11}(s)| \ge |b_{22}(s)| \ge \dots \ge |b_{kk}(s)| \quad \text{for all } s \in S,
\]
which completes the proof.
\end{proof}

\subsection{Decomposability via Haagerup--Schultz projections}
\label{subsec:Decomp}

An operator $A$ on a Hilbert space ${\mathcal H}$ is said to be {\em decomposable} if,
for every pair $(U,V)$ of open sets in the complex plane whose union is the whole complex plane, there are closed, $A$-invariant subspaces ${\mathcal H}'$ and ${\mathcal H}''$
such that ${\mathcal H}'+{\mathcal H}''={\mathcal H}$ and the restrictions of $A$
to these have spectra contained in $U$ and $V$, respectively.

The following characterization, established by Dykema, Noles, and Zanin, explicitly connects the decomposability of an operator with spectral inclusions of its restrictions to Haagerup–Schultz invariant subspaces associated with open disks.
\begin{prop}\label{prop:disks} \cite[Proposition 3.1]{DNZ17}
Let $A\in\Mcal$.
Then the following are equivalent:
\begin{enumerate}[label=(\roman*),leftmargin=25pt]
\item\label{it:Tdec} $A$ is decomposable.
\item\label{it:sigmaG}  For every open disk $D$ in $\Cpx$,
\begin{align}
\sigma(AP(A,\overline{D}))&\subseteq \overline{D}, \label{eq:sigG} \\
\sigma(AP(A,\Cpx\setminus D))&\subseteq\Cpx\setminus D, \label{eq:sigGc} \\
\sigma((1-P(A,\overline{D}))A)&\subseteq \Cpx\setminus D, \label{eq:sig1-G} \\
\sigma((1-P(A,\Cpx\setminus D))A)&\subseteq \overline{D}, \label{eq:sig1-Gc}
\end{align}
where the spectra are computed in the compressions of $\Mcal$ by the projections $P(A,\overline{D})$, $P(A,\Cpx\setminus D)$,
$1-P(A,\overline{D})$ and $1-P(A,\Cpx\setminus D)$, respectively.
\end{enumerate}
\end{prop}

\subsection{Spectral operators}

Spectral operators were introduced by Dunford \cite{Dun54} as a generalization of normal operators. Normal operators have a associated projection-valued spectral measure which behaves well with respect to the spectrum. Spectral operators are those which have a well-behaved idempotent-valued spectral measure. To be precise,
 
\begin{defi}
An operator $T$ on a Hilbert space $\HEu$ is said to be {\em spectral} if there exists a map $E$ which sends every Borel subsets of
$\Cpx$ to an idempotent operator in $\HEu$, which satisfies the following criteria:
 
 \begin{enumerate}
     \item $ E(\Cpx) = 1 $
     \item $ E(B_1 \cap B_2) = E(B_1) E(B_2)$, for all Borel sets $B_1,B_2$.
     \item $ E( \cup_{i=1}^\infty B_i ) = \sum_{i=1}^\infty E(B_i)$ whenever $B_i$ are pairwise disjoint Borel sets in $\Cpx$, and the sum on the right converges in the strong operator topology.
     \item $\sup\{\norm{E(B)} : B \text{ Borel in } \Cpx\} < \infty$.
     \item $E(B)T=TE(B)$ for all Borel sets $B$.
     \item For all Borel subsets $B$ of $\Cpx$, the spectrum of $T$ restricted to the range of $E(B)$ is contained in the closure of $B$:
        $$ \sigma( T \restrict _{E(B) \HEu}) \subseteq \overline{B}. $$
 \end{enumerate}

\end{defi}

\begin{defi}
A spectral operator $D \in \mathscr{B}(\mathscr{H})$ is said to be \emph{scalar-type} if $D = \int_{\mathbb{C}} \lambda \, \mathrm{d} E_D(\lambda)$, where $E_D$ is the idempotent-valued resolution of the identity for $D$.
\end{defi}

The following theorem gives the canonical reduction of bounded spectral operators, which we shall refer to as the Dunford decomposition of a spectral operator.

\begin{thm} \cite[Theorem 8]{Dun54} \label{Theorem: Spectral decom}
An operator $T \in \mathscr{B}(\mathscr{H})$ is spectral if and only if there is a scalar-type operator $D$ and a quasinilpotent operator $Q$, in $\mathscr{B}(\mathscr{H})$ such that $DQ = QD$ and $T = D + Q$. Furthermore, this decomposition is unique, and $T$ and $D$ have identical spectra and identical idempotent-valued spectral resolutions.
\end{thm}

\begin{rem}\label{Remark: Scalar-type}
The Dunford decomposition of a spectral operator may be viewed as a generalization of the Jordan-Chevalley decomposition of matrices, as scalar-type operators are similar to normal operators (see \cite[Theorem XV.6.4]{DS88}). More explicitly, if $D \in \mathscr{B}(\mathscr{H})$ is a scalar-type operator, there is a normal operator $N$ and an invertible operator $S$, in $\mathscr{B}(\mathscr{H})$ such that $D = S^{-1}NS$. It follows that the spectral projection of $N$ corresponding to Borel set $B$ is given by
\[
P_N(B) = S E_D(B) S^{-1},
\]
where $e_D$ denotes the idempotent-valued spectral resolution of $D$.
\end{rem}

\section{Main Results}

In this section, we assume $\mathcal{M}$ is a type $\mathrm{I}_k$ von Neumann algebra acting on a separable Hilbert Space $\mathcal{H}$. By Section \ref{subsec:UT}, $\mathcal{M} \cong M_{k}(C(S)) \cong M_{k}(L^{\infty}(X, \mu))$ for some Stonean space $S$, and a measurable space $(X,\mu)$. From Lemma \ref{inc order in two spaces} and Proposition \ref{lem:re-ord}, for $A \in \mathcal{M} $, there exists a unitary $U \in \mathcal{M}$ such that
\begin{equation}\label{A upper triang}
    A = U^* T U
\end{equation}
where $T$ is a $k \times k$ upper triangular matrix with entries in $L^{\infty}(X, \mu)$ such that the diagonal entries are non-decreasing almost everywhere, i.e.
\begin{equation}\label{eq:1st triang rep}
    T = (t_{ij})_{1 \le i, j \le k} \quad \text{where} \quad t_{ij} = 0 \text{ if } i > j
\end{equation}
and 
\begin{equation}\label{diag ordering}
     \lvert t_{11} \rvert \le \lvert t_{22} \rvert \le \dots \le \rvert t_{kk} \rvert \text{ a.e.}
\end{equation}

Fix $s \in [0,\infty)$, we partition $X$ into $(k+1)$ measurable sets $X_i$ for $0 \le i \le k$. Now, for $1 \le i \le k-1$ we define 
\begin{equation}\label{setdefn}
    X_i := \{x \in X \mid \lvert t_{ii}(x) \rvert \le s <  \lvert t_{i+1,i+1}(x) \rvert \},
\end{equation}
for $i=0$,
\[
X_0 := \{x \in X \mid s < |t_{11}(x)|\}
\]
and for $i=k$, 
\[
X_k := \{x \in X \mid |t_{kk}(x)| \leqslant s\}.
\]
Note that this partition depends on the operator $A$ and a scalar $s \ge 0$. We define the associated projection
\begin{equation*}\label{Qprimedefn}
    \begin{aligned}
   Q'(A, s) &= \sum_{i=1}^k \mathbbm{1}_{X_i} \cdot \left( \sum_{j=1}^i E_{jj} \right)\\
   &= \begin{pmatrix}
   \mathbbm{1}_{\sqcup_{i=1}^{k} X_i} & & & 0 \\
   & \mathbbm{1}_{\sqcup_{i=2}^{k} X_i} & & \\
   & & \ddots & \\
   0 & & & \mathbbm{1}_{X_k}
   \end{pmatrix}.
   \end{aligned}
\end{equation*}
Since $X_i$ is measurable for all $0 \le i \le k$, it follows that $Q'(A, s) \in \mathcal{M}$.

\begin{defi}\label{Q defi}
For any $s \ge 0$, we define the projection
\begin{equation*}
    Q(A, s) = U^* Q'(A, s) U \in \mathcal{M},
\end{equation*}
\end{defi}
and we will show that these projections are $A$-invariant.
\begin{lemma}\label{inv proj family}
For every $s \ge 0$, $Q(A, s)$ is an invariant projection of $A \in \mathcal{M}$. Consequently, $I - Q(A, s)$ is an invariant projection of $A^* \in \mathcal{M}$.
\end{lemma}

\begin{proof}
Let $Q'(A,s)$ be defined as in \eqref{Qprimedefn}. Since $A = U^* T U$ as in \eqref{A upper triang}, where $T = (t_{jl})_{1 \le j,l \le k}$ is an upper-triangular operator matrix over $L^{\infty}(X, \mu)$, $t_{jl} = 0$ whenever $j > l$.

We first evaluate the product $(I - Q'(A,s)) T Q'(A,s)$. Writing $Q'(A,s)$ and $I - Q'(A,s)$ explicitly as diagonal operator matrices in terms of characteristic functions yields
\[
I - Q'(A,s) = \operatorname{diag}\left(\mathbbm{1}_{\sqcup_{i=0}^0 X_i}, \, \mathbbm{1}_{\sqcup_{i=0}^1 X_i}, \, \dots, \, \mathbbm{1}_{\sqcup_{i=0}^{k-1} X_i}\right)
\]
and
\[
Q'(A,s) = \operatorname{diag}\left(\mathbbm{1}_{\sqcup_{i=1}^k X_i}, \, \mathbbm{1}_{\sqcup_{i=2}^k X_i}, \, \dots, \, \mathbbm{1}_{X_k}\right).
\]

Taking the $(j,l)$-entry of the matrix product $(I - Q'(A,s)) T Q'(A,s)$, we obtain
\[
\left( (I - Q'(A,s)) T Q'(A,s) \right)_{jl} = \mathbbm{1}_{\sqcup_{i=0}^{j-1} X_i} \, t_{jl} \, \mathbbm{1}_{\sqcup_{i=l}^k X_i}.
\]
For $j \le l$, the indicator functions correspond to disjoint sets, forcing $\mathbbm{1}_{\sqcup_{i=0}^{j-1} X_i} \mathbbm{1}_{\sqcup_{i=l}^k X_i} = 0$. For $j > l$, $t_{jl} = 0$ by the upper-triangularity of $T$. Thus, every entry in $(I - Q'(A,s)) T Q'(A,s)$ vanishes, establishing
\[
(I - Q'(A,s)) T Q'(A,s) = 0 \implies T Q'(A,s) = Q'(A,s) T Q'(A,s).
\]

Using Definition \ref{Q defi} and applying the unitary transformation $Q(A,s) = U^* Q'(A,s) U$, we have
\begin{align*}
A Q(A,s) &= U^* T U \left( U^* Q'(A,s) U \right) \\
&= U^* T Q'(A,s) U \\
&= U^* Q'(A,s) T Q'(A,s) U \\
&= Q(A,s) A Q(A,s).
\end{align*}
Hence, $Q(A,s)$ is an invariant projection for $A$. Taking the adjoint yields $(I - Q(A,s)) A^* = (I - Q(A,s)) A^* (I - Q(A,s))$, which shows that $I - Q(A,s)$ is invariant under $A^*$.
\end{proof}
The following elementary lemma ensures that upper-triangular matrices over $L^\infty(X,\mu)$ whose diagonal entries are bounded away from zero are invertible within $M_m(L^\infty(X,\mu))$.
\begin{lemma}\label{lemma:inverse}
Let $B \in M_m(L^{\infty}(X, \mu))$ where $B = [b_{ij}]$. If $B$ is an upper triangular matrix and $\text{ess inf } |b_{ii}| \ge c > 0$ for all $i$, then $B$ has an inverse in $M_m(L^{\infty}(X, \mu))$.
\end{lemma}
\begin{proof}
We use induction on $m$. For $m=1$, $B = b_{11} \in L^{\infty}(X, \mu)$ with $\text{ess inf } |b_{11}| \ge c$. Then $\frac{1}{b_{11}} \in L^{\infty}(X, \mu)$, so $B^{-1} = \frac{1}{b_{11}}$.
Assume the claim is true for $m-1$. Let $B \in M_m(L^{\infty}(X, \mu))$ be represented as
\begin{equation*}
    B = \begin{pmatrix} B' & v_0 \\ 0 & b_{mm} \end{pmatrix}
\end{equation*}
where $B' \in M_{m-1}(L^{\infty}(X, \mu))$ is upper triangular and satisfies the assumptions in the lemma, $ v_0 \in (L^{\infty}(X, \mu))^{ (m-1)}$ and $b_{mm} \in L^{\infty}(X, \mu)$. The inverse $B^{-1}$ is defined as
\begin{equation*}
    B^{-1} := \begin{pmatrix} (B')^{-1} & -(B')^{-1} v_0 \frac{1}{b_{mm}} \\ 0 & \frac{1}{b_{mm}} \end{pmatrix} \in M_m(L^{\infty}(X, \mu)).
\end{equation*}
It is easily verified that $B^{-1} B = B B^{-1} = I_m$. Hence, the claim holds.
\end{proof}
The following proposition establishes that the constructed projection $Q(A,s)$ restricts $A$ to an invariant subspace where its spectrum is localized inside the closed disk $s\overline{\mathbb{D}}$.
\begin{prop}\label{spectrum 1}
    For any $s \ge 0$,
\begin{enumerate}[label=(\roman*),leftmargin=25pt]
    \item $\sigma(A Q(A, s)) \subseteq \overline{s\mathbb{D}}$,
    \item $\sigma((I - Q(A, s)) A) \subseteq \mathbb{C} \setminus s\mathbb{D}$,
\end{enumerate}
where the spectra of $A Q(A, s)$ and $(I - Q(A, s)) A$ are considered as elements of the algebras $Q(A, s) \mathcal{M} Q(A, s)$ and $(I-Q(A, s)) \mathcal{M} (I-Q(A, s))$, respectively.
\end{prop}
\begin{proof}
Let $\lambda \notin \overline{s\mathbb{D}}$. Next,
\begin{align*}
    A Q(A, s) - \lambda Q(A, s) &= U^* T Q'(A, s) U - \lambda U^* Q'(A, s) U \\
    &= U^* \big((T - \lambda) Q'(A, s) \big)U.
\end{align*}
By the definition of $Q'(A,s)$ in \eqref{Qprimedefn}, 
\[
   (T - \lambda) Q'(A, s) = \sum_{i=1}^k (T - \lambda) \mathbbm{1}_{X_i} \cdot \left( \sum_{j=1}^i E_{jj} \right).
\]
The right hand side in the above equation can be written in the matrix form as
\begin{equation}\label{TQ}
    \sum_{i=1}^k \begin{pmatrix} (T-\lambda)^{(i)}|_{X_i} & 0 \\ 0 & 0 \end{pmatrix},
\end{equation}
where for each $1 \le i \le k$, $(T - \lambda)^{(i)}\big|_{X_i} \in M_i\!\left(L^\infty(X_i, \mu|_{X_i})\right)$ denotes the upper-triangular matrix obtained by taking the leading principal $i \times i$ submatrix of $T - \lambda$ and restricting its entries to the measurable subset $X_i \subseteq X$. The diagonal entries of $(T - \lambda)^{(i)}\big|_{X_i}$ are $\big\{ (t_{jj} - \lambda)\big|_{X_i} : 1 \le j \le i \big\}$. By \eqref{diag ordering} and  \eqref{setdefn}, $|t_{jj}| \le s$ \text{ a.e.} for all $1 \le j \le i$. We observe that
\[
\essinf_{x\in X_i} \left| t_{jj} (x)- \lambda \right| \ge \text{dist}(\lambda, \overline{s\mathbb{D}})
\]
for every $1 \le j \le i$. Let $d = \text{dist}(\lambda, \overline{s\mathbb{D}}) > 0$. Since $\essinf_{x\in X_i} \left| t_{jj} (x)- \lambda \right| \ge d > 0$, Lemma \ref{lemma:inverse} implies that $(T - \lambda)^{(i)}\big|_{X_i}$ is invertible in $M_i\!\left(L^\infty(X_i, \mu|_{X_i})\right)$.
Let
\begin{equation}\label{TQinv}
    {(T-\lambda)}^{(inv)} := \sum_{i=1}^k  \begin{pmatrix} \big((T-\lambda)^{(i)}|_{X_i}\big)^{-1} & 0 \\ 0 & 0 \end{pmatrix} \in M_k(L^{\infty}(X, \mu)).
\end{equation}
Then by using \eqref{TQ} and \eqref{TQinv}
\begin{align*}
    {(T-\lambda)}^{(inv)} \big((T - \lambda) Q'(A, s)\big) &= \sum_{i=1}^k \begin{pmatrix} (T-\lambda)^{(i)}|_{X_i} & 0 \\ 0 & 0 \end{pmatrix} \begin{pmatrix} \big((T-\lambda)^{(i,i)}|_{X_i}\big)^{-1} & 0 \\ 0 & 0 \end{pmatrix} \\
    &= \sum_{i=1}^k \begin{pmatrix} I_{i}|_{X_i} & 0 \\ 0 & 0 \end{pmatrix} \\
    &=\sum_{i=1}^k \mathbbm{1}_{X_i} \cdot \left( \sum_{j=1}^i E_{jj} \right) \\
    &= Q'(A, s),
\end{align*}
which implies $A Q(A, s) - \lambda Q(A, s)$ is invertible in $Q(A, s) \mathcal{M} Q(A, s)$.
Therefore, $\lambda \notin \sigma(A Q(A, s))$, proving $\sigma(A Q(A, s)) \subseteq \overline{s\mathbb{D}}$. 
Similarly, let $\lambda' \in {s\mathbb{D}}$. Next,
\begin{align*}
    (I - Q(A, s)) A - \lambda' (I - Q(A, s)) &= U^* (I-Q'(A, s)) T  U - \lambda' U^* (I - Q'(A, s)) U \\
    &= U^* (I - Q'(A, s)) (T - \lambda') U.
\end{align*}
Again by the definition of $Q'(A,s)$ in \eqref{Qprimedefn},
\[
   (1 - Q'(A, s)) (T - \lambda') = \sum_{i=0}^{k-1} \mathbbm{1}_{X_i} \cdot \left( \sum_{j=i+1}^k E_{jj} \right) (T - \lambda').
\]
The right hand side in the above equation can be written in the matrix form as
\begin{equation}\label{TQ1}
    \sum_{i=0}^{k-1} \begin{pmatrix} 0 & 0 \\ 0 & (T-\lambda')_{(k-i)}\big|_{X_i} \end{pmatrix},
\end{equation}
where for each $0 \le i \le k-1$, $(T-\lambda')_{(k-i)}\big|_{X_i} \in M_{k-i}\!\left(L^\infty(X_i, \mu|_{X_i})\right)$ denotes the upper-triangular matrix obtained by taking the trailing principal $(k-i) \times (k-i)$ submatrix of $T - \lambda'$ and restricting its entries to the measurable subset $X_i \subseteq X$. The diagonal entries of $(T-\lambda')_{(k-i)}\big|_{X_i}$ are $\big\{ (t_{jj} - \lambda)\big|_{X_i} : i+1 \le j \le k \big\}$. By \eqref{diag ordering} and  \eqref{setdefn}, $|t_{jj}| > s$ \text{ a.e.}  for all $i+1 \le j \le k$. We observe that
\[
\essinf_{x\in X_i} \left| t_{jj} (x)- \lambda' \right| \ge \text{dist}(\lambda', \mathbb{C} \setminus s\mathbb{D})
\]
for every $i+1 \le j \le k$. Let $\tilde{d} = \text{dist}(\lambda',\mathbb{C} \setminus s\mathbb{D}) > 0$. Since $\essinf_{x\in X_i} \left| t_{jj} (x)- \lambda' \right| \ge \tilde{d} > 0$, Lemma \ref{lemma:inverse} implies that $(T-\lambda')_{(k-i)}\big|_{X_i}$ is invertible in $M_{k-i}\!\left(L^\infty(X_i, \mu|_{X_i})\right)$. Let
\begin{equation}\label{TQinv1}
    {(T-\lambda)}_{(inv)} := \sum_{i=0}^{k-1} \begin{pmatrix}  0 & 0 \\ 0 & (T - \lambda')_{(k-i)}|_{X_i})^{-1} \end{pmatrix} \in M_k(L^{\infty}(X, \mu)).
\end{equation}
Then by using \eqref{TQ1} and \eqref{TQinv1}
\begin{align*}
    {(T-\lambda)}_{(inv)} \big((I - Q'(A, s)) (T - \lambda')\big) &= \sum_{i=0}^{k-1} \begin{pmatrix} 0 & 0 \\ 0 & (T-\lambda')_{(k-i)}\big|_{X_i} \end{pmatrix} \begin{pmatrix}  0 & 0 \\ 0 & (T - \lambda')_{(k-i)}|_{X_i})^{-1} \end{pmatrix} \\
    &= \sum_{i=0}^{k-1} \begin{pmatrix} 0 & 0 \\ 0 & I_{k-i}|_{X_i} \end{pmatrix} \\
    &=\sum_{i=0}^{k-1} \mathbbm{1}_{X_i} \cdot \left( \sum_{j=i+1}^k E_{jj} \right) \\
    &= I - Q'(A, s),
\end{align*}
which implies $ (I - Q(A, s)) A - \lambda' (I - Q(A, s))$ is invertible in $(I-Q(A, s)) \mathcal{M} (I-Q(A, s))$. Therefore, $\lambda' \notin \sigma((I - Q(A, s)) A)$, proving $\sigma((I - Q(A, s)) A) \subseteq \mathbb{C} \setminus s\mathbb{D}$. \\
\end{proof}
Let $P(A, \overline{s\mathbb{D}})$ denote the Haagerup-Schultz projection of $A$ onto the closed disk $\overline{s\mathbb{D}}$. We now compare $Q(A, s)$ directly to the Haagerup–Schultz projection $P(A, \overline{s \mathbb{D}})$.
\begin{prop}\label{comparison proj}
For any $s > 0$ and $\varepsilon > 0$, the following comparison holds:
\begin{equation*}
    Q(A, s) \le P(A, \overline{s\mathbb{D}}) \le Q(A, s+\varepsilon).
\end{equation*}
\end{prop}

\begin{proof}
By Proposition \ref{spectrum 1}, the restriction of $A$ to the invariant subspace $\mathcal{H}_0 =  Q(A, s)\mathcal{H}$, denoted by $A|_{\mathcal{H}_0}$, satisfies $\sigma(A|_{\mathcal{H}_0}) \subseteq  \overline {s\mathbb{D}}$. The spectral radius formula then implies
\begin{equation*}
    \lim_{n \to \infty} \|(A|_{\mathcal{H}_0})^n\|^{1/n} = \lim_{n \to \infty} \|(A^n|_{\mathcal{H}_0})\|^{1/n}\le s.
\end{equation*}
For any non-zero $\xi \in {\mathcal{H}_0} = Q(A, s)\mathcal{H}$, since $\lim_{n \to \infty} \|\xi\|^{1/n} = 1$, we obtain
\begin{equation*}
\begin{aligned}
\limsup_{n \to \infty} \|A^n \xi\|^{1/n}
&\le \Big( \lim_{n \to \infty}  \|(A^n|_{\mathcal{H}_0})\|^{1/n} \Big)
\Big( \lim_{n \to \infty} \|\xi\|^{1/n} \Big)\\
&\le s .
\end{aligned}
\end{equation*}
By Definition \ref{HS E and F definition}$(i)$ and Lemma \ref{HS E and F result}, it follows that $\xi \in P(A, \overline{s\mathbb{D}})\mathcal{H}$, which proves the first inequality $Q(A, s) \le P(A, \overline{s\mathbb{D}})$. 

Next, consider the invariant subspace $\mathcal{H}_1 = (I - Q(A, s + \varepsilon))\mathcal{H}$ of $A^*$. Proposition \ref{spectrum 1} and taking the adjoint implies that the restriction $A^*|_{\mathcal{H}_1}$ satisfies $\sigma(A^*|_{\mathcal{H}_1}) \subseteq \mathbb{C} \setminus (s + \varepsilon)\mathbb{D}$.
By the Spectral Mapping Theorem for inverses,
\[
\sigma\left((A^*|_{\mathcal{H}_1})^{-1}\right) \subseteq \frac{1}{s + \varepsilon}\overline{\mathbb{D}}.
\]
Applying the spectral radius formula yields
\[
\lim_{n \to \infty} \left\| (A^*|_{\mathcal{H}_1})^{-n} \right\|^{\frac{1}{n}} \le \frac{1}{s + \varepsilon}.
\]
For any $\eta \in {\mathcal{H}_1}=(I - Q(A, s+\varepsilon))\mathcal{H}$ with
$ \eta_n := (A^*|_{\mathcal{H}_1})^{-n} \eta $, the norm estimate gives
\[
\begin{aligned}
\limsup_{n \to \infty} \|\eta_n\|^{\frac{1}{n}} &\le \Big(\lim_{n \to \infty} \left\| (A^*|_{\mathcal{H}_1})^{-n} \right\|^{\frac{1}{n}} \Big) \Big( \lim_{n \to \infty}  \|\eta\|^{\frac{1}{n}} \Big) \\
&\le \frac{1}{(s+\varepsilon)}.
\end{aligned}
\]
Definition \ref{HS E and F definition} then gives $\eta \in F(A^*,s + \varepsilon)$.~ Lemma \ref{HS E and F relation} implies $E(A, s) \perp F(A^*, s+ \varepsilon)$, so $\eta \in E(A, s)^{\perp}=(I - P(A,  \overline{s\mathbb{D}}))\mathcal{H}$. This establishes $P(A, \overline{s \mathbb{D}}) \le Q(A, s+\varepsilon)$, completing the proof.
\end{proof}
By taking the limit as $\varepsilon \to 0^+$ and applying the comparison inequalities established in the above Proposition, we transfer the spectral bounds directly to the Haagerup--Schultz projection $P(A, s\overline{\mathbb{D}})$.
\begin{cor} \label{cor main}
    For any $s > 0$:
\begin{enumerate}[label=(\roman*),leftmargin=25pt]
    \item $\sigma\big(A \, P(A, \overline{s\mathbb{D}})\big) \subseteq \overline{s\mathbb{D}}$,
    \item  $\sigma((I -  P(A, \overline{s\mathbb{D}}) A) \subseteq \mathbb{C} \setminus s\mathbb{D}$,
\end{enumerate}
where the spectra of $A \,  P(A, \overline{s\mathbb{D}})$ and $(I -  P(A, \overline{s\mathbb{D}})) A$ are considered as elements of the algebras $ P(A, \overline{s\mathbb{D}}) \mathcal{M}  P(A, \overline{s\mathbb{D}})$ and $ (I- P(A, \overline{s\mathbb{D}})) \mathcal{M} (I- P(A, \overline{s\mathbb{D}}))$, respectively.
\end{cor}
\begin{proof}
    Combining Lemma \ref{p,q spectrum proj}, Proposition \ref{spectrum 1}$(i)$, and Proposition \ref{comparison proj}, we see that for any $s > 0$ and $\varepsilon > 0$,
\[
\sigma\big(A \, P(A, s\overline{\mathbb{D}})\big) \subseteq \sigma\big(A \, Q(A, s+\varepsilon)\big) \subseteq \overline{(s+\varepsilon)\mathbb{D}}.
\]
Letting $\varepsilon \rightarrow 0^+$, we obtain that $\sigma\big(A P(A, \overline{s \mathbb{D}})\big) \subseteq \overline{s \mathbb{D}}$.
Similarly, using Lemma \ref{p,q spectrum proj}, Proposition \ref{spectrum 1}$(ii)$, and Proposition \ref{comparison proj}, for any $s > 0$,
\[
\sigma\big((I - P(A, \overline{s \mathbb{D}}))A\big) \subseteq \sigma\big((I - Q(A, s))A\big) \subseteq \mathbb{C} \setminus s\mathbb{D}.
\]
\end{proof}
Furthermore, we establish analogous results for the Haagerup-Schultz projections of $A$ associated with $\mathbb{C} \setminus s\mathbb{D}$ for every $s \ge 0$. We consider the matrix representation of $A$ from Lemma \ref{inc order in two spaces} and Corollary \ref{cor: re-ord}, that there exists a unitary $ \widetilde{U} \in \mathcal{M}$ such that
\begin{equation*}
    A =  \widetilde{U}^* \widetilde{T} \widetilde{U},
\end{equation*}
where $\widetilde{T}$ is a $k \times k$ upper triangular matrix with entries in $L^{\infty}(X, \mu)$ such that the diagonal entries are non-decreasing almost everywhere, i.e.
\begin{equation*}\label{eq:1st triang rep&}
    T = (\widetilde{t}_{ij})_{1 \le i, j \le k} \quad \text{where} \quad \widetilde{t}_{ij} = 0 \text{ if } i > j
\end{equation*}
and 
\begin{equation*}\label{diag ordering&}
     \lvert \widetilde{t}_{11} \rvert \ge \lvert \widetilde{t}_{22} \rvert \ge \dots \ge \rvert \widetilde{t}_{kk} \rvert \text{ a.e.}
\end{equation*}
Fix $s \in [0,\infty)$, we again partition $X$ into $(k+1)$ measurable sets $\widetilde{X_i}$ for $0 \le i \le k$. For $1 \le i \le k-1$ we define 
\begin{equation}\label{setdefn&}
\widetilde{X}_i := \left\{ x \in X \;\middle|\; \lvert \widetilde{t}_{i,i}(x) \rvert > s \geq \lvert \widetilde{t}_{i+1,i+1}(x) \rvert \right\},
\end{equation}
for $i=0$,
\[
\widetilde{X}_0 := \{x \in X \mid s \ge |\widetilde{t}_{11}(x)|\}
\]
and for $i=k$,
\[
\widetilde{X}_k := \{x \in X \mid |\widetilde{t}_{kk}(x)| > s\}.
\]
Note that this partition as well depends on the operator $A$ and a scalar $s \ge 0$. We define the associated projections
\begin{equation*}\label{Qprimedefn&}
    \begin{aligned}
\widetilde{Q}'(A, s) &= \sum_{i=0}^{k} \mathbbm{1}_{\widetilde{X}_i} \cdot \left( \sum_{j=1}^{i} E_{jj} \right)\\
    &= \begin{pmatrix}
\mathbbm{1}_{\sqcup_{i=1}^{k} \widetilde{X}_i} & & & 0 \\
& \mathbbm{1}_{\sqcup_{i=2}^{k} \widetilde{X}_i} & & \\
& & \ddots & \\
0 & & & \mathbbm{1}_{\widetilde{X}_k}
\end{pmatrix}.
\end{aligned}
\end{equation*}
Since $\widetilde{X}_i$ is measurable for all $0 \leq i \leq k$, it follows that $\widetilde{Q}'(A, s) \in \mathcal{M}$. 
We define
\begin{equation*}
    \widetilde{Q}(A, s) =  \widetilde{U}^* \widetilde{Q}'(A, s)  \widetilde{U} \in \mathcal{M},
\end{equation*}
and establish that these projections are $A$-invariant.
\begin{lemma}\label{inv proj family&}
    For every $s \ge 0$, $\widetilde{Q}(A, s)$ is an invariant projection of $A \in \mathcal{M}$. Consequently, $I - \widetilde{Q}(A, s)$ is an invariant projection of $A^* \in \mathcal{M}$.
\end{lemma}
The proof proceeds analogously to that of Lemma \ref{inv proj family} by substituting $Q'(A,s), \, Q(A,s), \, U,$ and $T$ for $\widetilde{Q}'(A, s), \, \widetilde{Q}(A, s), \, \widetilde{U},$ and $\widetilde{T}$, respectively.\\

The following proposition establishes that the constructed projection $\widetilde{Q}(A,s)$ restricts $A$ to an invariant subspace where its spectrum is localized outside the open disk $s{\mathbb{D}}$.
\begin{prop}\label{spectrum 1&}
For any $s \geq 0$, the following inclusions hold:
\begin{enumerate}[label=(\roman*),leftmargin=25pt]
    \item $\sigma\big(A \, \widetilde{Q}(A, s) \big) \subseteq \mathbb{C} \setminus s\mathbb{D}$,
    \item $\sigma\big((I - \widetilde{Q}(A, s))A\big) \subseteq \overline{s\mathbb{D}}$,
\end{enumerate}
where the spectra of $A \widetilde{Q}(A, s)$ and $(I - \widetilde{Q}(A, s)) A$ are considered as elements of the algebras $\widetilde{Q}(A, s) \mathcal{M} \widetilde{Q}(A, s)$ and $(I-\widetilde{Q}(A, s)) \mathcal{M} (I-\widetilde{Q}(A, s))$, respectively.
\end{prop}
\begin{proof}
The proof of this proposition follows a similar argument to the Proposition \ref{spectrum 1}. Let $\lambda \in {s\mathbb{D}}$. Next,
\begin{equation*}
    \begin{aligned}
    A \widetilde{Q}(A, s) - \lambda \widetilde{Q}(A, s) &= \widetilde{U}^* \widetilde{T} \widetilde{Q}'(A, s)  \widetilde{U} - \lambda  \widetilde{U}^* \widetilde{Q}'(A, s)  \widetilde{U} \\
    &=  \widetilde{U}^* \big((\widetilde{T} - \lambda) \widetilde{Q}'(A, s) \big) \widetilde{U}\\
    &= \widetilde{U}^* \Big( \sum_{i=1}^k (\widetilde{T} - \lambda) \mathbbm{1}_{\widetilde{X}_i} \cdot \big( \sum_{j=1}^i E_{jj} \big) \Big) \widetilde{U}
\end{aligned}
\end{equation*}
The last in the above equation can be written in the matrix form as
\begin{equation}\label{TQ&}
    \widetilde{U}^* \Bigg(\sum_{i=1}^k \begin{pmatrix} (\widetilde{T}-\lambda)^{(i)}|_{\widetilde{X}_i} & 0 \\ 0 & 0 \end{pmatrix} \Bigg) \widetilde{U},
\end{equation}
where for each $1 \le i \le k$, $(\widetilde{T} - \lambda)^{(i)}\big|_{\widetilde{X}_i} \in M_i\!\left(L^\infty(\widetilde{X}_i, \mu|_{\widetilde{X}_i})\right)$ denotes the upper-triangular matrix obtained by taking the leading principal $i \times i$ submatrix of $\widetilde{T} - \lambda$ and restricting its entries to the measurable subset $\widetilde{X}_i \subseteq X$. The diagonal entries of $(\widetilde{T} - \lambda)^{(i)}\big|_{\widetilde{X}_i}$ are $\big\{ (\widetilde{t}_{jj} - \lambda)\big|_{\widetilde{X}_i} : 1 \le j \le i \big\}$.By \eqref{diag ordering&} and  \eqref{setdefn&}, $|\widetilde{t}_{jj}| > s$ \text{ a.e.} for all $1 \le j \le i$. We observe that
\[
\essinf_{x\in \widetilde{X}_i} \left| \widetilde{t}_{jj} (x)- \lambda \right| \ge \text{dist}(\lambda, \mathbb{C} \setminus s\mathbb{D})
\]
for every $1 \le j \le i$. Let $d = \text{dist}(\lambda, \mathbb{C} \setminus s\mathbb{D}) > 0$. Since $\essinf_{x\in \widetilde{X}_i} \left| \widetilde{t}_{jj} (x)- \lambda \right| \ge d > 0$, Lemma \ref{lemma:inverse} implies that $(\widetilde{T}- \lambda)^{(i)}\big|_{\widetilde{X}_i}$ is invertible in $M_i\!\left(L^\infty(\widetilde{X}_i, \mu|_{\widetilde{X}_i})\right)$.
Let
\begin{equation}\label{TQinv&}
    {(\widetilde{T}-\lambda)}^{(inv)} := \sum_{i=1}^k  \begin{pmatrix} \big((\widetilde{T}-\lambda)^{(i)}|_{\widetilde{X}_i}\big)^{-1} & 0 \\ 0 & 0 \end{pmatrix} \in M_k(L^{\infty}(X, \mu)).
\end{equation}
Then by using \eqref{TQ&} and \eqref{TQinv&}
\begin{align*}
    {(\widetilde{T}-\lambda)}^{(inv)}  \big((\widetilde{T} - \lambda) \widetilde{Q}'(A, s)\big) &= \sum_{i=1}^k \begin{pmatrix} (\widetilde{T}-\lambda)^{(i)}|_{\widetilde{X}_i} & 0 \\ 0 & 0 \end{pmatrix} \begin{pmatrix} \big((\widetilde{T}-\lambda)^{(i)}|_{\widetilde{X}_i}\big)^{-1} & 0 \\ 0 & 0 \end{pmatrix} \\
    &= \sum_{i=1}^k \begin{pmatrix} I_{i}|_{\widetilde{X}_i} & 0 \\ 0 & 0 \end{pmatrix} \\
    &=\sum_{i=1}^k \mathbbm{1}_{\widetilde{X}_i} \cdot \left( \sum_{j=1}^i E_{jj} \right) \\
    &= \widetilde{Q}'(A, s),
\end{align*}
which implies $A \widetilde{Q}(A, s) - \lambda \widetilde{Q}(A, s)$ is invertible in $\widetilde{Q}(A, s) \mathcal{M} \widetilde{Q}(A, s)$.
Therefore, $\lambda \notin \sigma(A \widetilde{Q}(A, s))$, proving $\sigma(A \widetilde{Q}(A, s)) \subseteq \mathbb{C} \setminus s\mathbb{D}$. 
Similarly, let $\lambda' \notin \overline{s\mathbb{D}}$. Next,
\begin{equation*}
    \begin{aligned}
    (I - \widetilde{Q}(A, s)) A - \lambda' (I - \widetilde{Q}(A, s)) &=  \widetilde{U}^* (I-\widetilde{Q}'(A, s)) \widetilde{T}   \widetilde{U} - \lambda'  \widetilde{U}^* (I - \widetilde{Q}'(A, s)) \widetilde{U} \\
    &=  \widetilde{U}^* (I - \widetilde{Q}'(A, s)) (\widetilde{T} - \lambda')  \widetilde{U}\\
    &=\widetilde{U}^* \Bigg( \sum_{i=0}^{k-1} \mathbbm{1}_{\widetilde{X}_i} \cdot \left( \sum_{j=i+1}^k E_{jj} \right) (\widetilde{T} - \lambda') \Bigg) \widetilde{U}^*
\end{aligned}
\end{equation*}
The last line in the above equation can be written in the matrix form as
\begin{equation*}
    \widetilde{U}^* \Bigg(\sum_{i=0}^{k-1} \begin{pmatrix} 0 & 0 \\ 0 & (\widetilde{T}-\lambda')_{(k-i)}\big|_{\widetilde{X}_i} \end{pmatrix} \Bigg) \widetilde{U},
\end{equation*}
where for each $0 \le i \le k-1$, $(\widetilde{T}-\lambda')_{(k-i)}\big|_{\widetilde{X}_i} \in M_{k-i}\!\left(L^\infty(\widetilde{X}_i, \mu|_{\widetilde{X}_i})\right)$ denotes the upper-triangular matrix obtained by taking the trailing principal $(k-i) \times (k-i)$ submatrix of $\widetilde{T} - \lambda'$ and restricting its entries to the measurable subset $\widetilde{X}_i \subseteq X$. The diagonal entries of $(\widetilde{T}-\lambda')_{(k-i)}\big|_{\widetilde{X}_i}$ are $\big\{ (\widetilde{t}_{jj} - \lambda)\big|_{\widetilde{X}_i} : i+1 \le j \le k \big\}$. By \eqref{diag ordering&} and  \eqref{setdefn&}, $|\widetilde{t}_{jj}| \le s$ \text{ a.e.}  for all $i+1 \le j \le k$. We observe that
\[
\essinf_{x\in \widetilde{X}_i} \left| \widetilde{t}_{jj} (x)- \lambda' \right| \ge \text{dist}(\lambda', \overline{s\mathbb{D}})
\]
for every $i+1 \le j \le k$. Let $\tilde{d} = \text{dist}(\lambda', \overline{s\mathbb{D}}) > 0$. Since $\essinf_{x\in X_i} \left| \widetilde{t}_{jj} (x)- \lambda' \right| \ge \tilde{d} > 0$, Lemma \ref{lemma:inverse} implies that $(\widetilde{T}-\lambda')_{(k-i)}\big|_{\widetilde{X}_i}$ is invertible in $M_{k-i}\!\left(L^\infty(\widetilde{X}_i, \mu|_{\widetilde{X}_i})\right)$. Let
\begin{equation*}
    {(\widetilde{T}-\lambda')}_{(inv)} := \sum_{i=0}^{k-1} \begin{pmatrix}  0 & 0 \\ 0 & (\widetilde{T} - \lambda')_{(k-i)}|_{\widetilde{X}_i})^{-1} \end{pmatrix} \in M_k(L^{\infty}(X, \mu)).
\end{equation*}
Then by using \eqref{TQ1} and \eqref{TQinv1}
\begin{align*}
    {(\widetilde{T}-\lambda')}_{(inv)} \big((I - \widetilde{Q}'(A, s)) (\widetilde{T} - \lambda')\big) &= \sum_{i=0}^{k-1} \begin{pmatrix} 0 & 0 \\ 0 & (\widetilde{T}-\lambda')_{(k-i)}\big|_{\widetilde{X}_i} \end{pmatrix} \begin{pmatrix}  0 & 0 \\ 0 & (\widetilde{T} - \lambda')_{(k-i)}|_{\widetilde{X}_i})^{-1} \end{pmatrix} \\
    &= \sum_{i=0}^{k-1} \begin{pmatrix} 0 & 0 \\ 0 & I_{k-i}|_{\widetilde{X}_i} \end{pmatrix} \\
    &=\sum_{i=0}^{k-1} \mathbbm{1}_{\widetilde{X}_i} \cdot \left( \sum_{j=i+1}^k E_{jj} \right) \\
    &= I - \widetilde{Q}'(A, s),
\end{align*}
 which implies $ (I - \widetilde{Q}(A, s)) A - \lambda' (I - \widetilde{Q}(A, s))$ is invertible in $(I-\widetilde{Q}(A, s)) \mathcal{M} (I-\widetilde{Q}(A, s))$. Therefore, $\lambda' \notin \sigma((I - \widetilde{Q}(A, s)) A)$, proving $\sigma((I - \widetilde{Q}(A, s)) A) \subseteq  \overline{s\mathbb{D}}$. \\
\end{proof}

Let $P(A, \mathbb{C} \setminus {s\mathbb{D}})$ denote the Haagerup-Schultz projection of $A$ onto the closed disk $\mathbb{C} \, \setminus \, {s\mathbb{D}}$. We similarly compare $\widetilde{Q}(A, s)$ directly to the Haagerup–Schultz projection $P(A, \mathbb{C} \setminus {s\mathbb{D}})$.

\begin{prop}\label{comparison proj&}
    For any $s > 0$ and $\varepsilon > 0$, the following comparison holds:
\begin{equation*}
    \widetilde{Q}(A, s) \le P(A, \mathbb{C} \setminus {s\mathbb{D}}) \le \widetilde{Q}(A, \max\{s-\varepsilon, 0\}).
\end{equation*}
\end{prop}
\begin{proof}
By Proposition \ref{spectrum 1&}, the restriction of $A$ to the invariant subspace $\mathcal{H}_0 =  \widetilde{Q}(A, s)\mathcal{H}$, denoted by $A|_{\mathcal{H}_0}$, satisfies $\sigma(A|_{\mathcal{H}_0}) \subseteq  \mathbb{C} \setminus {s\mathbb{D}}$. By the Spectral Mapping Theorem for inverses,
\[
\sigma\left((A|_{\mathcal{H}_0})^{-1}\right) \subseteq \frac{1}{s}\,\overline{\mathbb{D}}.
\] The spectral radius formula then implies
\begin{equation*}
    \lim_{n \to \infty} \|(A|_{\mathcal{H}_0})^{-n}\|^{1/n} = \lim_{n \to \infty} \|(A^n|_{\mathcal{H}_0})\|^{1/n}\le \frac{1}{s}.
\end{equation*}
For any non-zero $\eta \in {\mathcal{H}_0} = \widetilde{Q}(A, s)\mathcal{H}$ with $ \eta_n := (A|_{\mathcal{H}_0})^{-n} \eta $, since $\lim_{n \to \infty} \|\eta\|^{1/n} = 1$, we obtain
\[
\begin{aligned}
\limsup_{n \to \infty} \|\eta_n\|^{\frac{1}{n}} &\le \Big(\lim_{n \to \infty} \left\| (A|_{\mathcal{H}_0})^{-n} \right\|^{\frac{1}{n}} \Big) \Big( \lim_{n \to \infty}  \|\eta\|^{\frac{1}{n}} \Big) \\
&\le \frac{1}{s}.
\end{aligned}
\]
By Definition \ref{HS E and F definition}$(ii)$ and Lemma \ref{HS E and F result}, it follows that $\eta \in P(A, \mathbb{C} \setminus {s\mathbb{D}})$, which proves the first inequality $\widetilde{Q}(A, s) \le P(A, \mathbb{C} \setminus {s\mathbb{D}})$. 

Next, consider the invariant subspace $\mathcal{H}_1 = (I - \widetilde{Q}(A, \max\{s-\varepsilon, 0\}))\mathcal{H}$ of $A^*$. Proposition \ref{spectrum 1&} and taking the adjoint implies that the restriction $A^*|_{\mathcal{H}_1}$ satisfies $\sigma(A^*|_{\mathcal{H}_1}) \subseteq \overline{(\max\{s-\varepsilon, 0\})\mathbb{D}}$.
Applying the spectral radius formula yields
\[
\lim_{n \to \infty} \left\| (A^*|_{\mathcal{H}_1})^{n} \right\|^{\frac{1}{n}} \le \max\{s-\varepsilon, 0\}.
\]
For any $\xi \in {\mathcal{H}_1}=(I - \widetilde{Q}(A, \max\{s-\varepsilon, 0\}))\mathcal{H}$, the norm estimate gives
\[
\begin{aligned}
\limsup_{n \to \infty} \|(A^*)^n\xi\|^{\frac{1}{n}} &\le \Big(\lim_{n \to \infty} \left\| (A^*|_{\mathcal{H}_1})^{n} \right\|^{\frac{1}{n}} \Big) \Big( \lim_{n \to \infty}  \|\xi\|^{\frac{1}{n}} \Big) \\
&\le \max\{s-\varepsilon, 0\}.
\end{aligned}
\]
Definition \ref{HS E and F definition} then gives $\xi \in E(A^*, \max\{s-\varepsilon, 0\})$.~ Lemma \ref{HS E and F relation} implies $E(A^*, \max\{s-\varepsilon, 0\}) \perp F(A, s)$, so $\xi \in E(A^*, \max\{s-\varepsilon, 0\})^{\perp} \subseteq (I - P(A,  \mathbb{C} \setminus{s\mathbb{D}}))\mathcal{H}$. This establishes $P(A, \mathbb{C} \setminus {s \mathbb{D}}) \le \widetilde{Q}(A, \max\{s-\varepsilon, 0\})$, completing the proof.
\end{proof}
By taking the limit as $\varepsilon \to 0^+$ and applying the comparison inequalities established in the above Proposition, we transfer the spectral bounds directly to the Haagerup--Schultz projection $P(A, {\mathbb{C} \setminus s\mathbb{D}})$.
\begin{cor} \label{cor main&}
    For any $s > 0$:
\begin{enumerate}[label=(\roman*),leftmargin=40pt]
   \item $\sigma\big(A \, P(A, \mathbb{C} \setminus s\mathbb{D})\big) \subseteq \mathbb{C} \setminus s\mathbb{D}$,
    \item  $\sigma\big((I -  P(A, \mathbb{C} \setminus s\mathbb{D})) A\big) \subseteq \overline{s\mathbb{D}}$,
\end{enumerate}
where the spectra of $A \, P(A, \mathbb{C} \setminus s\mathbb{D})$ and $(I -  P(A, \mathbb{C} \setminus s\mathbb{D})) A$ are considered as elements of the algebras $  P(A, \mathbb{C} \setminus s\mathbb{D}) \mathcal{M}   P(A, \mathbb{C} \setminus s\mathbb{D})$ and $ (I -  P(A, \mathbb{C} \setminus s\mathbb{D})) \mathcal{M} (I -  P(A, \mathbb{C} \setminus s\mathbb{D}))$, respectively.
\end{cor}
\begin{proof}
    Combining Lemma \ref{p,q spectrum proj}, Proposition \ref{spectrum 1&}$(i)$, and Proposition \ref{comparison proj&}, we see that for any $s > 0$ and $\varepsilon > 0$,
\[
\sigma\big(A \, P(A, \mathbb{C} \setminus s\mathbb{D})\big) \subseteq \sigma\big(A \, \widetilde{Q}(A, \max\{s-\varepsilon, 0\})\big) \subseteq \mathbb{C} \setminus {\max\{s-\varepsilon, 0\}\mathbb{D}}.
\]
Letting $\varepsilon \rightarrow 0^+$, we obtain that $\sigma\big(A \, P(A, \mathbb{C} \setminus {s \mathbb{D}})\big) \subseteq \mathbb{C} \setminus {s \mathbb{D}}$.
Similarly, using Lemma \ref{p,q spectrum proj}, Proposition \ref{spectrum 1&}$(ii)$, and Proposition \ref{comparison proj&}, for any $s > 0$,
\[
\sigma\big((I - P(A, \mathbb{C} \setminus {s \mathbb{D}}))A\big) \subseteq \sigma\big((I - \widetilde{Q}(A, s))A\big) \subseteq  \overline{s\mathbb{D}}.
\]
\end{proof}
Finally, we prove the main result of this article.

\begin{proof}[Proof of Theorem \ref{main theorem}]
We use Proposition \ref{prop:disks} to show that $A$ is a decomposable operator. Let $D$ be an open disk. Then there exists $\lambda \in \mathbb{C}$ and $s > 0$ such that
$D = \lambda + s\mathbb{D}.$ Using Lemma \ref{lem:basicHS} $(iii)$ we have 
\[
P(A, \overline{D}) = P(A, \lambda + s\overline{\mathbb{D}}) = P(A - \lambda, s\overline{\mathbb{D}}).
\]
By Corollary \ref{cor main}$(i)$,
\begin{align*}
\sigma\big(A P(A, \overline{D})\big) &= \lambda + \sigma\big((A - \lambda) P(A, \overline{D})\big) \\
&= \lambda + \sigma\big((A - \lambda) P(A - \lambda, s\overline{\mathbb{D}})\big) \\
&\subseteq \lambda + s\overline{\mathbb{D}} \\
&= \overline{D}.
\end{align*}
This proves condition \eqref{eq:sigG} of Proposition \ref{prop:disks}.~Similarly, condition \eqref{eq:sigGc}, \eqref{eq:sig1-G} and \eqref{eq:sig1-Gc} of Proposition \ref{prop:disks} follows from Corollary \ref{cor main&}$(i)$, \ref{cor main}$(ii)$, and \ref{cor main&}$(ii)$, respectively. Hence, $A$ is decomposable.
\end{proof}

\section{Illustrative Examples}

In this section, we present two illustrative examples. Firstly, we construct an operator in a $\mathrm{I}_3$ von Neumann algebra that is not a Spectral operator.
\begin{example}
    Let $X = \{0\} \cup \left\{ \frac{1}{n} : n \in \mathbb{N} \right\}$ with the discrete $\sigma$-algebra. Define a probability measure $\mu$ on $X$ by 
\[
\mu(\{0\}) = \frac{1}{2} \quad \text{and} \quad \mu\left(\left\{\frac{1}{n}\right\}\right) = \frac{1}{2^{n+1}} \quad \text{for each } n \in \mathbb{N}.
\]
Consider the measure space $(X, \mu)$, and let $A \in M_3(L^\infty(X, \mu))$ be defined pointwise by
\[
A(x) = \begin{bmatrix} x & 1 & 0 \\ 0 & x & 1 \\ 0 & 0 & 0 \end{bmatrix}, \quad x \in X.
\]
Suppose, to the contrary, that $A$ is a spectral operator. By Theorem~\ref{Theorem: Spectral decom}, $A$ admits a unique decomposition $A = S + T$, where $S$ is a scalar-type operator and $T$ is a quasinilpotent operator commuting with $S$. 

Since $\mu$ is a purely atomic measure, the decomposition holds pointwise almost everywhere (hence everywhere on $X$); that is, for every $x \in X$, $S(x)$ is a scalar-type matrix and $T(x)$ is a quasinilpotent matrix. According to \cite[Theorem 3.4]{NS25}, the unique Dunford  decomposition of the $3 \times 3$ matrix $A(x)$ into its diagonal (scalar-type) and nilpotent parts is given by
\[
A(x) = \begin{bmatrix} x & 0 & -1/x \\ 0 & x & 1 \\ 0 & 0 & 0 \end{bmatrix} + \begin{bmatrix} 0 & 1 & 1/x \\ 0 & 0 & 0 \\ 0 & 0 & 0 \end{bmatrix}, \quad x \in X \setminus \{0\}.
\]
By uniqueness of the decomposition at each point $x \neq 0$, the nilpotent component $T(x)$ must equal 
\[
T(x) = \begin{bmatrix} 0 & 1 & 1/x \\ 0 & 0 & 0 \\ 0 & 0 & 0 \end{bmatrix}.
\]
However, as $x = 1/n$ ranges over $X \setminus \{0\}$, the set $\{1/x : x \in X \setminus \{0\}\} $ is unbounded, which implies that $\|T(x)\| \to \infty$ as $x \to 0$. Consequently, $T \notin M_3(L^\infty(X, \mu))$, contradicting the assumption that $T$ is a bounded operator. Thus, $A$ cannot be a spectral operator.
\end{example}

Next, we provide an example of an operator within the separable type $I_\infty$ factor, specifically $\mathcal{B}(\ell^2(\mathbb{N}))$, which fails to satisfy the norm convergence property.

\begin{example}
Let $S$ denote the unilateral shift operator on $\ell^2(\mathbb{N})$, defined by
\[ S(x_1, x_2, x_3, \dots) = (0, x_1, x_2, x_3, \dots). \]
For any $\lambda \in \mathbb{C}$ such that $|\lambda| < 1$, we define the vector
\[ x_\lambda := (1, \lambda, \lambda^2, \lambda^3, \dots) \in \ell^2(\mathbb{N}). \]
It is easily verified that $S^* x_\lambda = \lambda x_\lambda$. Consequently, $\lambda \in \sigma(S^*)$, which implies that the open unit disk $\mathbb{D}$ is contained in the spectrum $\sigma(S^*)$. Since the spectrum is a closed set, we have $\overline{\mathbb{D}} \subseteq \sigma(S^*)$.
Conversely, since $\|S\| = 1$, the spectral radius satisfies $r(S) \leq \|S\| = 1$, hence $\sigma(S) \subseteq \overline{\mathbb{D}}$. Therefore, we conclude that
\[ \sigma(S) = \sigma(S^*) =\overline{\mathbb{D}}. \]
We now examine the behavior of the adjoint $S^*$, which acts as the backward shift
\[ S^*(x_1, x_2, x_3, \dots) = (x_2, x_3, \dots). \]
Iterating this operator $n$ times yields
\[ S^{*n}(x_1, x_2, x_3, \dots) = (x_{n+1}, x_{n+2}, \dots). \]
The composition with its adjoint $S^n$ results in
\[ S^n S^{*n}(x_1, x_2, x_3, \dots) = (\underbrace{0, \dots, 0}_{n}, x_{n+1}, x_{n+2}, \dots). \]
This can be expressed as $S^n S^{*n} = I - P_n$, where $P_n$ denotes the orthogonal projection onto the subspace $\text{span}\{e_1, e_2, \dots, e_n\}$. Note that
\[ |(S^*)^n|^{1/n} = (S^n S^{*n})^{1/2n} = I - P_n. \]
As $n \to \infty$, the projections $P_n$ converge to the identity operator $I$ in the Strong Operator Topology (SOT). It follows that $|(S^*)^n|^{1/n}$ converge to the zero operator in the SOT.
Under these conditions, the Brown measure $\mu_{S^*}$ of $S^*$ is the Dirac mass at the origin, $\delta_0$. Consequently, we observe a strict inclusion
\[ \text{supp}(\mu_{S^*}) = \{0\} \subsetneq \sigma(S^*) = \overline{\mathbb{D}}. \]
By the criteria established in \cite[Theorem $4.3$]{DNZ17}, the fact that the support of the Brown measure is strictly contained within the spectrum implies that $S^*$ does not possess the norm convergence property.
\end{example}

\section{Acknowledgments}
I am deeply grateful to my advisor, Dr.~Ken Dykema, for his guidance, encouragement, and support throughout this work.~I would also like to thank \href{https://shekhawatrenu.github.io/ShekhawatRenu/index.html}{Dr.~Renu Shekhawat} for several valuable and stimulating discussions that helped shape my understanding of the ideas developed in this article.

\end{document}